\documentclass[amssymb,twoside,11pt]{article}
\usepackage{latexsym,amsmath,graphicx}
\usepackage[dvipsnames]{xcolor}
\usepackage{amsfonts}
\newtheorem{theorem}{Theorem}[section]
\newtheorem{Lemma}[theorem]{{\bf Lemma}}

\newtheorem{rem}[theorem]{{\bf Remark}}

\newtheorem{definition}{Definition}[section]
\numberwithin{equation}{section}
\newenvironment{proof}{\indent{\em Proof:}}{\quad \hfill
$\Box$\vspace*{2ex}}

\begin{document}
\setcounter{page}{1}
\begin{center}
\vspace{0.4cm} {\large{\bf Implicit Fractional Differential Equations in Banach Spaces via Picard and Weakly Picard Operator Theory }} \\
\vspace{0.5cm}
Sagar T. Sutar  $^{1}$\\
sutar.sagar007@gmail.com\\
\vspace{0.35cm}
Kishor D. Kucche $^{2}$ \\
kdkucche@gmail.com \\

\vspace{0.35cm}
$^1$ Department of Mathematics, Vivekanand College (Autonomous),\\ Kolhapur-416003, Maharashtra, India.\\
$^{2}$ Department of Mathematics, Shivaji University, Kolhapur-416 004, Maharashtra, India.\\
\end{center}

\def\baselinestretch{1.0}\small\normalsize

\begin{abstract}
In this paper, by employing fixed-point methods, we obtain
the existence and uniqueness results for the nonlinear implicit fractional differential equations in Banach spaces. Further, we obtain
the uniqueness, dependence of the solution on the initial condition as well as on the functions involved on the right-hand side by means of Picard and weakly Picard operator theory and Pompeiu--Hausdorff functional. 
\end{abstract}
\def\baselinestretch{1.5}
\allowdisplaybreaks
\section{Introduction}
The initial and boundary value problems for fractional differential equations (FDEs) are naturally appearing in the modeling of various physical phenomena that appearing in diverse disciplines of science and engineering \cite{KD1,WGG,RHIL,FMain,FMe,VET}.  Subsequently, this field is drawing the attention of many mathematicians for researching numerous types of FDEs for its theory and applicability. Basics of fractional calculus, detailed theoretical analysis and applications of FDEs can be found in the interesting monographs by Kilbas et al.\cite{Kilbas}, Lakshmikantham et al. \cite{VLAk}, Podlubny  \cite{Pod}, Diethelm  \cite{Diethelm}. Fundamental results on the existence and uniqueness of solutions, comparison results, different types of data dependency, the existence of extremal solutions and global existence of the nonlinear FDEs  can found in \cite{Laksh1,Laksh2,Laksh,Gejji1,Diethelm1,Agarwal1,JWang3}. Recently, there have been numerous papers that deal with the existence, uniqueness, controllability  and various qualitative properties of the solution for FDEs with initial and boundary conditions \cite{Ravi1,Ravi2,NDC,GIS}. 

The implicit differential equation of the form
$$
x^{(n)}(t)=f\left(t,\ x(t),\ x'(t),\ x''(t),\cdots,\ x^{(n-1)}(t)\right),
$$
with  different kind of initial or boundary condition has been analyzed for existence, uniqueness and various other qualitative properties of the solution, for instance, see \cite{DLI} and the references cited therein. Motivated by the investigations in the theory of integer order implicit differential equation considered above and its different special from, Benchohra et al. \cite{Benchora4,Benchora5} and Nieto et al. \cite{JNieto} have initiated the study of implicit FDEs of the form 
$$
D^{\alpha}x(t)= f(t,x(t), D^{\alpha}x(t)),
$$
with a different kind of initial and boundary condition.  Nieto et al. \cite{JNieto}  derived existence and uniqueness results by employing  fixed point theory and an approximation method.  Benchohra et al. {\cite{Benchora3,Ben2,Ben3}} investigated  the existence of integrable solutions for implicit fractional  functional differential equations with and without delay.  Kucche et.al. in  \cite{Kucche,Kucche1,Sutar} have analyzed implicit FDEs for  existence, uniqueness  dependence of  solution on various initial data and Ulam--Hyers and Ulam--Hyers-Rassias stabilities. 

On the other hand, Rus \cite{Rus1,Rus2} introduced the fundamental results of the weakly Picard operators theory and employed it to analyze the Cauchy problem, boundary value problem, integral equations, and difference equations. Wang  \cite{JWang7} obtained some existence, uniqueness, and data dependence results by means of Picard and weakly Picard operators theory and the Bielecki norms. Then again, Otrocol and Ilea \cite{ODI,ODI2} analyzed differential and integrodifferential equations with abstract Volterra operators using the method of Picard and weakly Picard operators.

Motivated by the work of \cite{Rus2,JWang7},  in the present  manuscript,  we obtain
the existence, uniqueness and dependency of solution for the
nonlinear implicit FDEs of the form,
\begin{align}
^c_{0}\mathcal{D}_{t}^\alpha x(t)&= f\left( t, x(t), ^c_{0}\mathcal{D}_{t}^\alpha x(t) \right), t\in J, ~0<\alpha<1 \label{Pe1.1} \\
x(0)&=x_0\in X, \label{Pe1.2}
\end{align}
in the  Banach space $X=(X,\|\cdot\|)$, where $J=[0,T], T>0$,  $^c_{0}\mathcal{D}_{t}^\alpha$ is the Caputo fractional derivative operator of order $\alpha$ and lower terminal $0$ and  $f:J\times X\times X\to X$ is a nonlinear function satisfying certain assumptions.

We obtain the  existence and uniqueness of the solution by using fixed-point methods in the space of Lipschitz function and considering the Bielecki  norm.  The dependency of a solution is obtained through Picard operator theory and  Pompeiu--Hausdorff functional. It is observed that the single inequality obtained for the difference of solutions via Picard operator theory gives simultaneously,  the  uniqueness of solution, the dependence of the solution on the initial condition as well as on the functions involved on the right-hand side of the problem \eqref{Pe4.1}-\eqref{Pe4.2}.  Whereas, the inequality obtained for the  solutions via Pompeiu--Hausdorff functional  gives the dependency of the  solution on  functions involved in the right hand side of the problem \eqref{Pe4.1}-\eqref{Pe4.2}. 

% % % % % % % % % % % % % % % % % % % % % % %

This paper is structured in the following way. In section 2, we present briefly the basic definitions
and results of fractional calculus, Picard and weakly Picard operators.  In Section 3, we obtain
the existence, uniqueness and dependency of solution for the
nonlinear implicit FDEs \eqref{Pe4.1}-\eqref{Pe4.2}. In Section 4, we provide examples to illustrate our results. Finally,we close the paper with concluding remarks.
   
\section{Preliminaries}
In this section, we give some definitions and basic results of Caputo fractional derivative  \cite{Diethelm,Pod} and Picard and weakly Picard operator theory  \cite{Rus1,JWang7,Rus}. 
\begin{definition}
 The Riemann-Liouville fractional integral~ $\mathcal{I}^{\beta}_{0,t}$~ of order $\beta > {0}$ of a function $g\in C[0,T]$, $T> 0$  is defined as
\begin{equation*}
 \mathcal{I}^{\beta}_{0,t}\; g(t)=\frac{1}{\Gamma(\beta)}\int_{0} ^t (t-\sigma)^{\beta -1 }g(\sigma)d\sigma ,\;\; t> 0,
\end{equation*}
provided the integral exists. 
\end{definition}

\begin{definition}
Let $0<\alpha\leq 1 $ then the Caputo fractional derivative $ ^c_{0}\mathcal{D}_{t}^\alpha$ of order $\alpha$ with lower terminal $0$ of a function $g\in C^{1}[0,T] $  is defined as
\begin{equation*}
^c_{0}\mathcal{D}_{t}^\alpha g(t)=\frac{1}{\Gamma(1-\alpha)}\int_{0}^t(t-\sigma)^{-\alpha}g'(\sigma)d\sigma,\;\; t>0.
\end{equation*}
\end{definition}

\noindent Next, we provide the basics of Picard and weakly Picard operators theory from  \cite{Rus1,JWang7,Rus}. Let $X:=\left(X,d \right)$ be a metric space and $T: X \to X$ an operator. Then:
  \begin{itemize}
\item [(i)] $P(X):=\left\lbrace Y\subseteq X : Y\ne \phi \right\rbrace $.
\item [(ii)] $F_T:=\left\lbrace x\in X: T(x)=x\right\rbrace $-- the fixed point set of $T$.
\item [(iii)] $I(T):=\left\lbrace Y\in P(X): T(Y)\subseteq Y \right\rbrace$ -- a set of $T$-invariant nonempty subsets $X$.
\item [(iv)] The Pompeiu-Hausdorff  functional $H_d:P(X)\times P(X)\to \mathbb{R}_+\cup\left\lbrace \infty\right\rbrace $ is defined as 
 $$H_d\left(Y,Z \right)=\max\left\lbrace\underset{a\in Y}{\sup}\;\underset{b\in Z}{\inf} d(a,b), \;\underset{b\in Z}{\sup}\;\underset{a\in Y}{\inf} d(a,b)\right\rbrace. $$\end{itemize}
 \begin{definition} [\cite{JWang7}]
 Let $(X,d)$ be a metric space. An operator $T:X\to X$ is a Picard operator if there exists $x^*\in X$ such that $F_T=\left\lbrace x^* \right\rbrace $ and the sequence $\left(T^n(x_0) \right)_{n\in\mathbb{N}}$ converges to $x^*$ for all $x_0\in X$.
 \end{definition}
 \begin{definition}[\cite{JWang7}]
 Let $(X,d)$ be a metric space. An operator $T:X\to X$ is a weakly Picard operator if the sequence $\left(T^n(x_0) \right)_{n\in\mathbb{N}}$ converges for all $x_0\in X$ and its limit ( which may depend on $x_0$ ) is a fixed point of $T$.
  \end{definition}

If $T:X\to X$ is a weakly Picard operator, then the  operator
$T^{\infty}:X\to X$ defined by  $$\;T^{\infty}(x)=\lim\limits_{n\to \infty}T^n(x).$$ 
 \begin{Lemma}[\cite{JWang7}]\label{Lm2.3}
   Let $(X,d)$ be a  metric space. Then $T:X\to X$ is a weakly Picard operator if and only if there exists a partition of $X$, $X=\bigcup_{\lambda\in \Lambda}X_\lambda$, where $\Lambda$ is the indices' set of partition, such that
\begin{itemize}
 \item [\normalfont(i)] $X_\lambda \in I(T)$,
\item [\normalfont(ii)]$T|_{X_\lambda}:X_\lambda\to X_\lambda$ is a Picard operator, for all $\lambda\in \Lambda$
\end{itemize}
\end{Lemma}
\begin{Lemma}[\cite{JWang7}]\label{PLm2.1}
Let $(X,d)$ be a complete metric space and $T,S:X\to X$ two operators. We suppose the following
\begin{itemize}
\item [\normalfont(i)]  $T$ is a contraction with constant $\alpha$ and $F_T=\left\lbrace x_T^* \right\rbrace $. 
\item [\normalfont(ii)] $S$ has a fixed points and  $ x_S^*\in F_S$.
\item [\normalfont(iii)] There exists $\gamma>0$ such that $d\left(T(x),S(x) \right)\leq \gamma$,\;for all $x\in X$. 
\end{itemize}
Then,
$$d\left( x_T^*,x_S^*\right)\leq\dfrac{\gamma}{1-\alpha}  $$
\end{Lemma}
\begin{Lemma}[\cite{JWang7}]\label{PLm2.2}
Let $(X,d)$ be a complete metric space and $T,S:X\to X$ two orbitally continuous operators. We suppose the following:
\begin{itemize}
\item [\normalfont(i)]  There exists $\alpha\in [0,1) $ such that
\begin{align*}
d\left( T^2(x),T(x)\right)&\leq \alpha\; d\left( T(x),x \right)\\d\left( S^2(x),S(x)\right)&\leq \alpha\; d\left( S(x),x \right)
\end{align*}
for all $x\in X$. 
\item [\normalfont(ii)] There exists $\gamma>0$ such that $d\left(T(x),S(x) \right)\leq \gamma$,\; for all $x\in X$.
\end{itemize}
Then,
 $$H_d\left( F_T,F_S\right)\leq\dfrac{\gamma}{1-\alpha},  $$ 
where $H_d$ denotes Pompeiu-Hausdorff functional.
    \end{Lemma}

\section{Main results}
To derive our main results, we consider the spaces defined in \cite{JWang7}.

Let $C:=C(J,X)$ be the space of all $X$-valued continuous functions defined on $J$. Then $ \mathcal{C}:=\left( C,\|\cdot\|_\mathcal{C} \right)$ is the Banach space with the Chebyshev norm $$
\|x\|_{\mathcal{C}}:=\sup_{t\in J}\left\lbrace\|x(t)\| \right\rbrace. $$
Further, $\mathcal{B}=\left( C,\|\cdot\|_B\right) $ is the Banach space endowed with  the Bielecki  norm 
$$
\|x\|_B:=\sup_{t\in J}\left\lbrace\dfrac{\|x(t)\|  }{\mathbb{E}_\alpha\left(\theta t^\alpha \right) }  \right\rbrace,\; x\in C\;\text{for some}\;\theta>0.
$$ 
Let  $d_\mathcal{B}$ and $d_\mathcal{C}$ are the metrics on $C$  generated by the norms $\|\cdot\|_\mathcal{B}$ and $\|\cdot\|_\mathcal{C}$. 

Let $L>0$ and define the set
$$C_L(J,X):=\left\lbrace x\in C(J,X):\|x(t_1)-x(t_2)\|\leq L|t_1-t_2|,\; \forall\; t_1,~t_2\in J\right\rbrace.$$ 

Let $R>0$,  $B_R=S[0,R]=\left\lbrace x\in X:\|x\|\leq R\right\rbrace$  and  
$$
C_L(J,B_R):=\left\lbrace x\in C(J,B_R):\|x(t_1)-x(t_2)\|\leq L|t_1-t_2|,\; \forall\; t_1,t_2\in J\right\rbrace.
$$

We acquire our main results via equivalent fractional integral equation to the nonlinear implicit FDEs \eqref{Pe1.1}--\eqref{Pe1.2} given in the following lemma.

\begin{Lemma}[\cite{Benchora3}]\label{lm3.1}
The solution of the IVP \eqref{Pe1.1}--\eqref{Pe1.2}  can be expressed by the integral equation
 \begin{equation}\label{Pe3.1}
 x(t)=x_0+\dfrac{1}{\Gamma(\alpha)}\int_{0}^{t}(t-s)^{\alpha-1}z(s)ds,\; t\in J;
 \end{equation} 
 where $z$ is the solution of the functional integral equation
 \begin{equation}\label{Pe3.2}
 z(t)=f\left( t,x_0+\dfrac{1}{\Gamma(\alpha)}\int_{0}^{t}(t-s)^{\alpha-1}z(s)ds, z(t)\right), t\in J.  
 \end{equation}
\end{Lemma}

\subsection{Existence and Uniqueness Results:}

\begin{theorem}\label{Thm3.2}
Assume that:
\begin{itemize}
\item[\normalfont{(A1)}] \;$f:J\times  X\times X\to X$ is  a continuous function and satisfies Lipschitz type condition
\begin{equation}
\begin{small}\label{e3.3}
\|f(t_1,x_1,y_1)-f(t_2,x_2,y_2)\|\leq M_1|t_1-t_2|+M_2\|x_1-x_2\|+M_3\|y_1-y_2\|,
\end{small}
\end{equation}
for all $t_i\in J$ and  $x_i,\;y_i \in X$,
 where $M_i>0,\;  i=1,2 $ and $ 0<M_3<1.$
\end{itemize}
Then, the problem \eqref{Pe1.1}--\eqref{Pe1.2} has unique solution, provided that \begin{equation}\label{e3.4}
\dfrac{M_2}{\Gamma(\alpha+1)}T^{\alpha}+M_3 <1.
\end{equation}
Further, the solution $z^*\in C_L(J,B_R)$ of functional integral equation \eqref{Pe3.2} can be obtained by successive approximation method starting from any element of $C_L(J,B_R)$. Utilizing this  $z^*$, the solution $x^*$ of \eqref{Pe1.1}--\eqref{Pe1.2} is then acquired from  the equation \eqref{Pe3.1}. 
\end{theorem}
\begin{proof}
In the view of condition \eqref{e3.4}, we can choose  $R>0$ such that \begin{equation}\label{e3.5}
\dfrac{M_2 \|x_0\|+K}{ 1-\left(\dfrac{M_2}{\Gamma(\alpha+1)}T^{\alpha}+M_3\right)} \leq R, 
\end{equation} where $K=\sup_{t\in J}\left\| f\left( t,0,0\right)\right\|$.
Further the condition $0<M_3<1$ allows us to choose  $L>0$ such that
\begin{equation}\label{e3.6}
\dfrac{  M_1+\dfrac{2M_2R}{\Gamma(\alpha+1)}} {1-M_3} \leq L.
\end{equation}
With these choices of $L$ and $R$ consider the space 
$$C_L(J,B_R):=\left\lbrace x\in C(J,B_R):\|x(t_1)-x(t_2)\|\leq L|t_1-t_2|,\; \forall\; t_1,t_2\in J\right\rbrace.$$
Define the mapping, 
\begin{equation}\label{e3.1}
(Tz)(t)=f\left( t,x_0+\dfrac{1}{\Gamma(\alpha)}\int_{0}^{t}(t-s)^{\alpha-1}z(s)ds, z(t)\right),\;  t\in J,\; z\in C_L(J,B_R),
\end{equation}
We show that the operator $T$ is a Picard operator on $C_L(J,B_R).$ We divide the proof in the two parts.

\noindent Part-(I):  First we prove that $T$ is self map on $C_L(J,B_R).$ We give the proof in the following steps.\\
\noindent Step~1:~ $Tz\in C(J,X)$ for any $z\in C_L(J,B_R)$.

\noindent Let any $\delta>0$. Then for any $z\in C_L(J,B_R)$ and $t\in J,$  we have
\begin{equation}\label{e3.8}
\|z(t+\delta)-z(t)\|\leq L|t+\delta-t|=L\delta,\;  \;\text{and}\; \|z(t)\|\leq R.
\end{equation}
By Lipschitz type condition on $f$ given in (A1) and using the inequalities in \eqref{e3.8}, we obtain
\begin{align*}
&\|(Tz)(t+\delta)-(Tz)(t)\|\\
&=\left\|f\left( t+\delta,x_0+\dfrac{1}{\Gamma(\alpha)}\int_{0}^{t+\delta}(t+\delta-s)^{\alpha-1}z(s)ds, z(t+\delta)\right)\right.\\ &\quad\quad\left.
\quad- f\left( t,x_0+\dfrac{1}{\Gamma(\alpha)}\int_{0}^{t}(t-s)^{\alpha-1}z(s)ds, z(t)\right)\right\|\\
&\leq M_1|(t+\delta)-t|+M_2\left\|\left( x_0+\dfrac{1}{\Gamma(\alpha)}\int_{0}^{t+\delta}(t+\delta-s)^{\alpha-1}z(s)ds\right)\right.\\ &\quad\quad\left.
\quad-\left( x_0+\dfrac{1}{\Gamma(\alpha)}\int_{0}^{t}(t-s)^{\alpha-1}z(s)ds\right)  \right\| +M_3\|z(t+\delta)-z(t)\|\\
&\leq M_1 \delta+\dfrac{M_2}{\Gamma(\alpha)}\int_{0}^{t}\left[ (t+\delta-s)^{\alpha-1}-(t-s)^{\alpha-1}\right]\|z(s)\|ds\\
&\quad+\dfrac{M_2}{\Gamma(\alpha)}\int_{t}^{t+\delta} (t+\delta-s)^{\alpha-1}\|z(s)\|ds+M_3\|z(t+\delta)-z(t)\|\\
&\leq (M_1+M_3L)\delta+\dfrac{M_2R}{\Gamma(\alpha)}\int_{0}^{t}\left[ (t+\delta-s)^{\alpha-1}-(t-s)^{\alpha-1}\right]ds
+\dfrac{M_2R}{\Gamma(\alpha)}\int_{t}^{t+\delta} (t+\delta-s)^{\alpha-1}ds\\
&=(M_1+M_3L)\delta+\dfrac{M_2R}{\Gamma(\alpha+1)}\left\lbrace  \left[-\delta^\alpha+(t+\delta)^\alpha-t^\alpha\right]+\delta^\alpha \right\rbrace \\
&=(M_1+M_3L)\delta+\dfrac{M_2R}{\Gamma(\alpha+1)}  \left[(t+\delta)^\alpha-t^\alpha\right].
\end{align*}
Therefore,
$$\lim\limits_{\delta\to 0} \|(Tz)(t+\delta)-(Tz)(t)\|= 0.$$
This proves  that  $Tz\in C(J,X)$.\\

\noindent Step~2:~$Tz\in C_L(J,B_R)$ for any $z\in C_L(J,B_R).$

\noindent Let any $z\in C_L(J,B_R)$. Then using inequalities \eqref{e3.5} and \eqref{e3.8}, for any $t\in J$, we obtain 
\begin{align*}
\|T(z)(t)\|&=\left\| f\left( t,x_0+\dfrac{1}{\Gamma(\alpha)}\int_{0}^{t}(t-s)^{\alpha-1}z(s)ds, z(t)\right)\right\|\\
&=\left\| f\left( t,x_0+\dfrac{1}{\Gamma(\alpha)}\int_{0}^{t}(t-s)^{\alpha-1}z(s)ds, z(t)\right)- f\left( t,0,0\right)\right\| +\left\| f\left( t,0,0\right)\right\| \\
&\leq M_2 \|x_0\|+\dfrac{M_2}{\Gamma(\alpha)}\int_{0}^{t}(t-s)^{\alpha-1}\|z(s)\|ds+M_3\|z(t)\|+K\\
& \leq (M_2 \|x_0\|+K)+\dfrac{M_2R}{\Gamma(\alpha)}\int_{0}^{t}(t-s)^{\alpha-1}ds+M_3R\\
&= (M_2 \|x_0\|+K)+\left[ \dfrac{M_2}{\Gamma(\alpha+1)}T^{\alpha}+M_3\right] R\\
&\leq R\left[1- \left( \dfrac{M_2}{\Gamma(\alpha+1)}T^{\alpha}+M_3\right)\right] +\left[ \dfrac{M_2 R}{\Gamma(\alpha+1)}T^{\alpha}+M_3R\right] = R.
\end{align*}
This proves  that  any $z\in C_L(J,B_R)$, $Tz\in C(J,B_R).$\\
\noindent Step~3:~ For any $z\in C(J,B_R),\;Tz$ is Lipschitz with Lipschitz constant $L$. 

\noindent Let  any $z\in C_L(J,B_R)$ and  $t_1,t_2\in J$ with $t_1<t_2$. Then by Lipschitz type condition on $f$ given in (A1)  and the inequality  \eqref{e3.8}, we obtain
\begin{align*}
&\|(Tz)(t_1)-(Tz)(t_2)\|\\
&=\left\| f\left( t_1,x_0+\dfrac{1}{\Gamma(\alpha)}\int_{0}^{t_1}(t_1-s)^{\alpha-1}z(s)ds, z(t_1)\right)\right.\\ &\quad\quad\left.
\quad-f\left( t_2,x_0+\dfrac{1}{\Gamma(\alpha)}\       nt_{0}^{t_2}(t_2-s)^{\alpha-1}z(s)ds, z(t_2)\right)\right\|\\
&\leq M_1|t_1-t_2|+\dfrac{M_2}{\Gamma(\alpha)}\int_{0}^{t_1}\left[ (t_1-s)^{\alpha-1}-(t_2-s)^{\alpha-1}\right] \|z(s)\|ds\\
&\qquad+\dfrac{M_2}{\Gamma(\alpha)}\int_{t_1}^{t_2}(t_2-s)^{\alpha-1}\|z(s)\|ds +M_3\|z(t_2)-z(t_1)\|\\
&\leq M_1|t_1-t_2|+\dfrac{M_2R}{\Gamma(\alpha)}\int_{0}^{t_1}\left[ (t_1-s)^{\alpha-1}-(t_2-s)^{\alpha-1}\right] ds\\
&\qquad+\dfrac{M_2R}{\Gamma(\alpha)}\int_{t_1}^{t_2}(t_2-s)^{\alpha-1}ds+M_3L|t_2-t_1|\\
&=\left( M_1+M_3L\right) |t_1-t_2|+\dfrac{M_2R}{\Gamma(\alpha+1)}\left[ (t_2-t_1)^{\alpha}+t_1^\alpha-t_2^\alpha\right]+\dfrac{M_2R}{\Gamma(\alpha+1)} (t_2-t_1)^{\alpha} \\
&\leq \left( M_1+M_3L\right) |t_1-t_2|+\dfrac{2M_2R}{\Gamma(\alpha+1)} (t_2-t_1)^{\alpha}=\left( M_1+\dfrac{2M_2R}{\Gamma(\alpha+1)}+M_3L \right) |t_2-t_1|\\
&\leq \left(L(1-M_3)+M_3L \right) |t_2-t_1|=L|t_2-t_1|.
\end{align*}
This shows that $Tz$ is Lipschitz function with Lipschitz constant $L$. From Steps-1 to 3, we get $Tz\in C_L(J,B_R)$. Hence $T$ is self mapping on $C_L(J,B_R)$.\\

\noindent Part-(II):  To prove $T:C_L(J,B_R)\to C_L(J,B_R)$  is a Picard Operator, we prove that $T$ is contraction. Let any $z_1,z_2\in C_L(J,B_R) $. Then for any $t\in J$, we have 
\begin{align*}
&\|T(z_1)(t)-T(z_2)(t)\|\\
&=\left\| f\left( t,x_0+\dfrac{1}{\Gamma(\alpha)}\int_{0}^{t}(t-s)^{\alpha-1}z_1(s)ds, z_1(t)\right)\right.\\ &\quad\quad\left.
\quad- f\left( t,x_0+\dfrac{1}{\Gamma(\alpha)}\int_{0}^{t}(t-s)^{\alpha-1}z_2(s)ds, z_2(t)\right)\right\| \\
&\leq \dfrac{M_2}{\Gamma(\alpha)}\int_{0}^{t}(t-s)^{\alpha-1}\|z_1(s)-z_2(s)\|ds+M_3\|z_1(t)-z_2(t)\|\\
&\leq \dfrac{M_2}{\Gamma(\alpha)}\int_{0}^{t}(t-s)^{\alpha-1}\mathbb{E}_\alpha(\theta s^\alpha)\sup_{s\in J}\left( \dfrac{\|(z_1-z_2)(s)\|}{\mathbb{E}_\alpha(\theta s^\alpha)}\right) 	 ds+M_3\mathbb{E}_\alpha(\theta t^\alpha)\sup_{t\in J}\left( \dfrac{\|(z_1-z_2)(t)\|}{\mathbb{E}_\alpha(\theta t^\alpha)}\right)\\
&=\dfrac{M_2\|z_1-z_2\|_B}{\Gamma(\alpha)}\int_{0}^{t}(t-s)^{\alpha-1}\mathbb{E}_\alpha(\theta s^\alpha)ds+M_3E_\alpha(\theta t^\alpha)\|z_1-z_2\|_B.
\end{align*}
But,
$$
I^\alpha_{0,t}E_\alpha(\theta t^\alpha)=\dfrac{1}{\theta}\left[ E_\alpha(\theta t^\alpha)-1\right]\leq \dfrac{1}{\theta} E_\alpha(\theta t^\alpha),~t\in J.
$$
This gives
$$\|(Tz_1)(t)-(Tz_2)(t)\|\leq\dfrac{M_2\|z_1-z_2\|_B}{\theta}\mathbb{E}_\alpha(\theta t^\alpha)+M_3E_\alpha(\theta t^\alpha)\|z_1-z_2\|_B,\; t\in J.
$$
Therefore,
\begin{align*}
\|Tz_1-Tz_2\|_B&=\sup_{t\in J}\left( \dfrac{\|(Tz_1-Tz_2)(t)\|}{\mathbb{E}_\alpha(\theta t^\alpha)}\right)\\ &\leq\left( \dfrac{M_2}{\theta}+M_3\right) \|z_1-z_2\|_B,\; t\in J.
\end{align*}
Since  $0<M_3<1$, we can choose sufficiently large  $\theta$, so that $ \dfrac{M_2}{\theta}+M_3 <1$. This proves that $T:C_L(J,B_R)\to C_L(J,B_R)$ is contraction and hence Picard operator. By contraction principle there is $z^*$ in $C_L(J,B_R)$ such that $z^*=Tz^*$. This $z^*$ is the unique solution of functional integral equation \eqref{Pe3.2}. Further for any $z\in C_L(J,B_R)$, $\|T^nz-z^*\|_B\to 0$  as $n\to\infty$. Substituting this $z^*$ in \eqref{Pe3.1}, we obtain the unique solution $x^*$ of the nonlinear implicit FDEs \eqref{Pe1.1}--\eqref{Pe1.2}.
\end{proof}
\subsection{Dependency of solution Through Picard Operator Theory: }
To investigate the  data dependency of the solution of the  nonlinear implicit FDEs \eqref{Pe1.1}--\eqref{Pe1.2},  we consider the another nonlinear implicit FDEs of the form
\begin{align}
^c_{0}\mathcal{D}_{t}^\alpha x(t)&= g\left( t,  x(t),\;^c_{0}\mathcal{D}_{t}^\alpha x(t) \right), t\in[0, T],\label{e5.1}  \\
x(0)&=y_0\in X, \label{e5.2}
\end{align}
where $g:J\times  X\times X\to X$ is  any nonlinear function need not equal to $f$.

 \begin{theorem}\label{ThmP3.3}
 Assume that the functions $f$ and $g$ satisfies the hypothesis \normalfont{(A1)}. Let there exists $\eta\in L^1(J,\mathbb{R}_+)\cap C(J,\mathbb{R}_+)$ such that
 $$\|f(t,x,y)-g(t,x,y)\|\leq \eta(t),\; t\in J,  ~x,y \in X.$$
If  the condition \eqref{e3.4} holds and  $\theta>0$ is such that $\dfrac{M_2}{\theta}+M_3<1$, then the  solution  $x^*$ of \eqref{Pe1.1}--\eqref{Pe1.2} and the  solution  $y^*$ of \eqref{e5.1} --\eqref{e5.2} satisfies the inequality 
\begin{align} \label{e311}
 \|x^*-y^*\|_B\leq \|x_0-y_0\|+\dfrac{K_\eta}{\theta\left[ 1-\left( \dfrac{M_2}{\theta}+M_3\right) \right] },
\end{align}
where $\;K_\eta=\max \left\lbrace \eta(t):t\in J\right\rbrace$.
 \end{theorem}
 \begin{proof} By Lemma ~\ref{lm3.1}, the equivalent integral equation to the nonlinear implicit FDEs \eqref{e5.1} --\eqref{e5.2} is given by 
 \begin{equation}\label{e5.3}
  x(t)=y_0+\dfrac{1}{\Gamma(\alpha)}\int_{0}^{t}(t-s)^{\alpha-1}z(s)ds,\; t\in J,
  \end{equation} 
  where $z$ is the solution of the functional integral equation
  \begin{equation}\label{e5.4}
  z(t)=g\left( t,y_0+\dfrac{1}{\Gamma(\alpha)}\int_{0}^{t}(t-s)^{\alpha-1}z(s)ds, z(t)\right), t\in J.  
  \end{equation}
   Since $g:J\times X\times X\to X$ satisfies the conditions of Theorem \ref{Thm3.2}, proceeding as in the proof of Theorem \ref {Thm3.2}, the condition $\eqref{e3.4}$ allows us to choose $\tilde{L}>0$ and $\tilde{R}>0$ such that the mapping $$S : C_{\tilde{L}}(J,B_{\tilde{R}})\to C_{\tilde{L}}(J,B_{\tilde{R}})$$ defined by
 $$S(z)(t)=g\left( t,y_0+\dfrac{1}{\Gamma(\alpha)}\int_{0}^{t}(t-s)^{\alpha-1}z(s)ds, z(t)\right), ~t\in J   $$
 has unique fixed point $F_S\in C_{\tilde{L}}(J,B_{\tilde{R}})$. 
 Define $L^* =\max\{L,\tilde{L}\}$ . Then one can verify that  $F_T,~F_S\in C_{L^*}(J,B_{R^*}).$ 
 
 Now, for any $z\in C_{L^*}(J,B_{R^*})$ and $t\in J$, we have
 \begin{align*}
& \|(Tz)(t)-(Sz)(t)\|
\\
&=\left\| f\left( t,x_0+\dfrac{1}{\Gamma(\alpha)}\int_{0}^{t}(t-s)^{\alpha-1}z(s)ds, z(t)\right)-g\left( t,y_0+\dfrac{1}{\Gamma(\alpha)}\int_{0}^{t}(t-s)^{\alpha-1}z(s)ds, z(t)\right)\right\| \\
&\leq \eta(t) \leq K_\eta.
 \end{align*}
 Since $\mathbb{E}_\alpha(\theta t^\alpha)\geq 1$ for all $t\in J,$ we have 
  $$\|Tz-Sz\|_B=\sup_{t\in J}\left( \dfrac{\|(Tz-Sz)(t)\|}{\mathbb{E}_\alpha(\theta t^\alpha)}\right)\leq K_\eta .$$
Since the operators $T,~S$ satisfies the conditions of Lemma \ref{PLm2.1}, by applying it, we obtain
 $$ \|F_T-F_S\|_B\leq\dfrac{ K_\eta}{1-\left( \dfrac{M_2}{\theta}+M_3\right)}.$$ 
 Note that  $$x^*(t)=x_0+\dfrac{1}{\Gamma(\alpha)}\int_{0}^{t}(t-s)^{\alpha-1}F_T(s)ds,\; t\in J;$$
 and
 $$y^*(t)=y_0+\dfrac{1}{\Gamma(\alpha)}\int_{0}^{t}(t-s)^{\alpha-1}F_S(s)ds,\; t\in J;$$
 are the unique solutions of nonlinear implicit FDEs  \eqref{Pe1.1}--\eqref{Pe1.2} and \eqref{e5.1} --\eqref{e5.2} respectively. Then for any $t\in J$, we have
 \begin{align*}
 \|x^*(t)-y^*(t)\|&\leq \|x_0-y_0\|+\dfrac{1}{\Gamma(\alpha)}\int_{0}^{t}(t-s)^{\alpha-1}\|F_T(s)-F_S(s)\|ds\\
  &\leq\|x_0-y_0\|+\dfrac{1}{\Gamma(\alpha)}\int_{0}^{t}(t-s)^{\alpha-1}\mathbb{E}_\alpha\left(\theta s^\alpha \right)\sup_{s\in J}\left( \dfrac{\|F_T(s)-F_S(s)\|}{\mathbb{E}_\alpha\left(\theta s^\alpha \right)}\right) ds\\
 &=\|x_0-y_0\|+\dfrac{\|F_T-F_S\|_B}{\Gamma(\alpha)}\int_{0}^{t}(t-s)^{\alpha-1}\mathbb{E}_\alpha\left(\theta s^\alpha \right)ds\\
 &=\|x_0-y_0\|+\|F_T-F_S\|_B\left( \dfrac{\mathbb{E}_\alpha\left(\theta t^\alpha \right)-1}{\theta}\right) \\
 &\leq\|x_0-y_0\|+\dfrac{ \mathbb{E}_\alpha\left(\theta t^\alpha \right)K_\eta}{\theta\left[ 1-\left( \dfrac{M_2}{\theta}+M_3\right)\right] }. 
 \end{align*}  
 This gives, 
 \begin{equation*}
\|x^*-y^*\|_B=\sup_{t\in J}\left( \dfrac{\|(x^*-y^*)(t)\|}{\mathbb{E}_\alpha\left(\theta t^\alpha \right)}\right)\leq \|x_0-y_0\|+\dfrac{ K_\eta}{\theta\left[ 1-\left( \dfrac{M_2}{\theta}+M_3\right)\right] },
 \end{equation*}
 which is the desired inequality \eqref{e311}.
 \end{proof}
\subsection*{Remark:} 
\begin{itemize}
\item[ (1)] Theorem \ref{ThmP3.3} gives the dependence of the solution of the problem \eqref{Pe1.1}-\eqref{Pe1.2} on the initial condition as well as on the functions involved on the right-hand side.
\item[ (2)] If $K_\eta=0$ in \eqref{e311}, that is when $f=g$, Theorem \ref{ThmP3.3} gives the dependency of solution of \eqref{Pe1.1}-\eqref{Pe1.2} on initial condition.
\item[ (3)] If $x_0=y_0$ and $K_\eta\ne0$ in \eqref{e311}, then Theorem \ref{ThmP3.3} gives the dependency of solution of \eqref{Pe1.1}-\eqref{Pe1.2} on functions involved in the right hand side of equation.
\item[(4)] If $x_0= y_0$ and $K_\eta=0$ in \eqref{e311}, Theorem \ref{ThmP3.3} gives the  uniqueness of solution of the problem  \eqref{Pe1.1}-\eqref{Pe1.2}.
\end{itemize}
\subsection{Dependency of solution through Pompeiu--Hausdorff functional} 
 Next, we consider an equation of the form,
 \begin{equation}\label{Pe4.1}
  x_z(t)=z(0)+\dfrac{1}{\Gamma(\alpha)}\int_{0}^{t}(t-s)^{\alpha-1}z(s)ds,\; t\in J;
  \end{equation} 
  where $z$ is a solution of fractional functional equation  
  \begin{equation}\label{Pe4.2}
  z(t)=f\left( t,z(0)+\dfrac{1}{\Gamma(\alpha)}\int_{0}^{t}(t-s)^{\alpha-1}z(s)ds, z(t)\right), t\in J,
  \end{equation} 
where $f$ is as in the problem \eqref{Pe1.1}-\eqref{Pe1.2}.
 From equations \eqref{Pe4.1}--\eqref{Pe4.2} it follows that $x_z(0)=z(0)=f(0,z(0),z(0))$
 \begin{theorem}\label{Thm4.1} Let $f:J\times X\times X\to X$ satisfies the assumptions  \normalfont{(A1)}  and $f(0,x,x)=x,$ for all $x\in X.$ Then the equation  \eqref{Pe4.1} has a solution in $C_L(J,B_R)$ for some $L,R>0$. If $\mathcal{S}\subset C_L(J,B_R)$ is its solution set, then card $\mathcal{S}$= card $B_R$.
 \end{theorem}
 \begin{proof}
 Consider the operator 
$T_*:\left( C_L(J,B_R),\|\cdot\|_B\right) \to \left( C_L(J,B_R),\|\cdot\|_B\right)  $
defined by,
\begin{equation}\label{eP3.10}
T_*(z)(t)=f\left( t,z(0)+\dfrac{1}{\Gamma(\alpha)}\int_{0}^{t}(t-s)^{\alpha-1}z(s)ds, z(t)\right),\; t\in J
\end{equation}
Proceeding as in proof of Theorem  \ref{Thm3.2}, there exist constants $L,R>0$  such that $T_*$  is  a continuous operator on $C_L(J,B_R)$. One can verify that, the  operator $T_*$ does not satisfies Lipschitz condition and  hence  $T_*$ is not a  Picard operator.

\noindent Next, We apply Lemma \ref{Lm2.3} to prove  that the operator $T_*$ is a  weakly Picard operator. For each $ \alpha\in B_R$, we define $X_\alpha=\left\lbrace z\in C_L(J,B_R) : z(0)=\alpha\right\rbrace $.
Clearly, $C_L(J,B_R)=\bigcup_{\alpha\in B_R}X_\alpha$. By assumption  for any $z\in X_\alpha,$   we have
$$T_*(z)(0)=f\left( 0,z(0), z(0)\right)=f\left( 0,\alpha,\alpha \right)=\alpha. $$ 
Therefore, for each $\alpha\in B_R$, $X_\alpha$ is invariant under $T_*$. Proceeding as in the proof of Theorem \eqref{Thm3.2}, one can verify that for each $\alpha\in B_R$,  the operator $T_*|_{X_\alpha}:X_\alpha\to X_\alpha$ is a Picard operator. Using   Lemma \ref{Lm2.3},  $T_*$ is   weakly Picard operator. By definition 2.2, for any $x_0\in C_L(J,B_R$ $) $ the seq $T_*^n(x_0)$ converges to the fixed point of $T^*$, which is the solution of fractional functional equation \eqref{Pe4.2}, from which we can obtain  the solution of \eqref{Pe4.1}.  Following the similar steps of Theorem 4.1 of  \cite{JWang7}, one can complete the proof of 
$$ \mbox{card}~\mathcal{S}= \mbox{card}~ B_R.$$  
\end{proof}
% % % % % % % % % % % %Remain % % % % % % %

In order to discuss the data dependency of solution, we consider the equations \eqref{Pe4.1}--\eqref{Pe4.2}, and the equations
\begin{equation}\label{Pe4.4}
 x_z(t)=z(0)+\dfrac{1}{\Gamma(\alpha)}\int_{0}^{t}(t-s)^{\alpha-1}z(s)ds,\; t\in J,
\end{equation} 
where $z$ is a solution of functional equation  
\begin{equation}\label{Pe4.5}
z(t)=g\left( t,\tilde{z}(0)+\dfrac{1}{\Gamma(\alpha)}\int_{0}^{t}(t-s)^{\alpha-1}z(s)ds, z(t)\right), t\in J,  
\end{equation} 
and  $g\in C(J\times X\times X, X) $ is any other nonlinear function need not equal to $f$. 
 \begin{theorem}\label{Thm4.2}
 Assume that $f$ and $g$ satisfies the hypothesis $\normalfont{(A1)}$ and $f(0,x,x)=x=g(0,x,x)$ for all $x\in X$. Let there exists $K >0$ such that
 \begin{align*}
 \|f(t,x,y)-g(t,x,y)\|\leq K\mathbb{E}_\alpha(\theta t^\alpha),\;  t\in J,\; x,y\in X,
 \end{align*}
for sufficiently large $\theta>0$ satisfying  $\dfrac{M_2}{\theta}+M_3<1$.
If  $\mathcal{S}$ and $\hat{\mathcal{S}}$ are the solution sets of the equation \eqref{Pe4.2} and \eqref{Pe4.4} respectively, then there exists $L^*,~ R^*>0$ such that
 
\begin{align}\label{kk}
H_{\|\cdot\|_B} \left(\mathcal{S},\hat{\mathcal{S}}\right)\leq \dfrac{KT^\alpha}{\Gamma(\alpha+1)\left[ 1-\left( \dfrac{M_2}{\theta}+M_3\right)\right]},
\end{align}
 where $H_{\|\cdot\|_B}$ is the Pompeiu-Hausdorff functional with respect to $\|\cdot\|_B$ on $C_{L^*}(J,B_{R^*})$.
 \end{theorem}
 \begin{proof}
Since $g:J\times X\times X\to X$ satisfies the conditions of Theorem \ref{Thm4.1}, proceeding as in the proof of Theorem \ref {Thm4.1}, we can choose $\tilde{L}>0$ and $\tilde{R}>0$ such that the mapping $$S_*:\left( C_{\tilde{L}}(J,B_{\tilde{R}}),\|\cdot\|_B\right) \to \left( C_{\tilde{L}}(J,B_{\tilde{R}}),\|\cdot\|_B\right)  $$
defined by
\begin{equation}\label{eP3.11}
S_*(z)(t)=g\left( t,z(0)+\dfrac{1}{\Gamma(\alpha)}\int_{0}^{t}(t-s)^{\alpha-1}z(s)ds, z(t)\right)
\end{equation}
has a fixed point $F_{S_*}\in C_{\tilde{L}}(J,B_{\tilde{R}}).$
Define $L^* =\max\{L,\tilde{L}\}$ and $R^*=\max\{R,\tilde{R}\}$, then one can verify that  $F_{T_*},F_{S_*}\in C_{L^*}(J,B_{R^*}).$  Now, for any $z \in  C_{L^*}(J,B_{R^*}) $ and  $t\in J$, we have 
 \begin{align*}
&\|T^2_*(z)(t) -T_*(z)(t)\|=\|T_*(T_*(z)(t)) -T_*z(t)\|\\
&=\left\|f\left( t,\, T_*(z)(0)+\dfrac{1}{\Gamma(\alpha)}\int_{0}^{t}(t-s)^{\alpha-1}T_*(z(s))ds,\,  T_*(z(t))\right) \right.\\
& \left. \qquad-f\left( t,z(0)+\dfrac{1}{\Gamma(\alpha)}\int_{0}^{t}(t-s)^{\alpha-1}z(s)ds, z(t)\right)\right\| \\
&\leq M_2\left\| \left( z(0)+\dfrac{1}{\Gamma(\alpha)}\int_{0}^{t}(t-s)^{\alpha-1}T_*(z(s))ds\right)- \left( z(0)+\dfrac{1}{\Gamma(\alpha)}\int_{0}^{t}(t-s)^{\alpha-1}z(s)ds\right)\right\| \\
&\qquad + M_3\|T_*(z)(t)-z(t)\|\\
&\leq \dfrac{M_2}{\Gamma(\alpha)}\int_{0}^{t}(t-s)^{\alpha-1}\|T_*(z(s))-z(s)\|ds+M_3\|T_*(z)(t)-z(t)\|\\
&\leq \dfrac{M_2}{\Gamma(\alpha)}\int_{0}^{t}(t-s)^{\alpha-1}\mathbb{E}_\alpha(\theta s^\alpha)\dfrac{\|T_*(z(s))-z(s)\|}{\mathbb{E}_\alpha(\theta s^\alpha)}ds+M_3\mathbb{E}_\alpha(\theta t^\alpha)\dfrac{\|T_*(z(t))-z(t)\|}{\mathbb{E}_\alpha(\theta t^\alpha)}\\
&\leq \dfrac{M_2}{\Gamma(\alpha)}\|T_*(z)-z\|_B\int_{0}^{t}(t-s)^{\alpha-1}\mathbb{E}_\alpha(\theta s^\alpha)ds+M_3\mathbb{E}_\alpha(\theta t^\alpha)\|T(z)-z\|_B\\
&\leq {M_2}\|T_*(z)-z\|_B\dfrac{\mathbb{E}_\alpha(\theta t^\alpha)}{\theta}+M_3\mathbb{E}_\alpha(\theta t^\alpha)\|T_*(z)-z\|_B.
\end{align*}
This gives
$$\dfrac{\|T^2_*z(t) -T_*z(t)\| }{\mathbb{E}_\alpha(\theta t^\alpha)} \leq \dfrac{M_2}{\theta}\|T_*z-z\|_B+M_3\|T_*z-z\|_B$$ 
Therefore
$$
\|T^2_*z -T_*z\|_B  \leq \dfrac{M_2}{\theta}\|T_*z-z\|_B+M_3\|T_*z-z\|_B=\left( \dfrac{M_2}{\theta}+M_3\right) \|T_*z-z\|_B,~ z\in C_{L^*}(J,B_{R^*}).
$$ 
On the similar line, we have
$$\|S^2_*z -S_*z\|_B  \leq \left( \dfrac{M_2}{\theta}+M_3\right) \|S_*z-z\|_B,~ z\in C_{L^*}(J,B_{R^*})$$ 
Further, for any $z \in  C_{L^*}(J,B_{R^*}) $ and  $t\in J$, we have 
\begin{align*}
&\|T_*(z)(t) -S_*(z)(t)\|\\
&=\left\|f\left( t,z(0)+\dfrac{1}{\Gamma(\alpha)}\int_{0}^{t}(t-s)^{\alpha-1}z(s)ds, z(t)\right)\right.\\
& \left. \qquad- g\left( t,z(0)+\dfrac{1}{\Gamma(\alpha)}\int_{0}^{t}(t-s)^{\alpha-1}z(s)ds, z(t)\right)\right\| \\
&\leq 
K \mathbb{E}_\alpha(\theta t^\alpha)
\end{align*}
This gives,
$$ \|T_*z -S_*z\|_B=\sup_{t\in J} \dfrac{\|T_*z(t) -S_*z(t)\|}{\mathbb{E}_\alpha(\theta t^\alpha)}\leq K,~ z\in C_{L^*}(J,B_{R^*}).$$ 
Since, the operators $T_*,\;S_*$ fulfills the requirements of Lemma \eqref{PLm2.2}, by an application of it, we obtain
$$H_{\|\cdot\|_B} (F_{T_*},F_{S_*})\leq \dfrac{K}{1-\left( \dfrac{M_2}{\theta}+M_3\right) }.$$
Note that  $$x_z^*(t)=z(0)+\dfrac{1}{\Gamma(\alpha)}\int_{0}^{t}(t-s)^{\alpha-1}F_{T_*}(s)ds,\; t\in J;$$
 and
 $$y_z^*(t)=z(0)+\dfrac{1}{\Gamma(\alpha)}\int_{0}^{t}(t-s)^{\alpha-1}F_{S_*}(s)ds,\; t\in J;$$
 are the  solutions of  \eqref{Pe4.1} and \eqref{Pe4.4}  respectively. Therefore, for any $t\in J$, we have
 \begin{align*}
 \|x_z^*(t)-y_z^*(t)\|&\leq \dfrac{1}{\Gamma(\alpha)}\int_{0}^{t}(t-s)^{\alpha-1}\|F_{T_*}(s)-F_{S_*}(s)\|ds \end{align*}  
 This gives, 
 \begin{align*}
  H_{\|\cdot\|_B}(x_z^*,y_z^*)&\leq \dfrac{1}{\Gamma(\alpha)}\int_{0}^{t}(t-s)^{\alpha-1}H_{\|\cdot\|_B} (F_{T_*},F_{S_*})ds\\
  &\leq  \dfrac{K}{\Gamma(\alpha)\left[ 1-\left( \dfrac{M_2}{\theta}+M_3\right)\right]  }\int_{0}^{t}(t-s)^{\alpha-1}ds\\
  &=\dfrac{Kt^\alpha}{\Gamma(\alpha+1)\left[ 1-\left( \dfrac{M_2}{\theta}+M_3\right)\right]}\\
  &\leq \dfrac{KT^\alpha}{\Gamma(\alpha+1)\left[ 1-\left( \dfrac{M_2}{\theta}+M_3\right)\right]}
  \end{align*} 
Therefore, 
\begin{align*}
H_{\|\cdot\|_B} \left(\mathcal{S},\hat{\mathcal{S}}\right)\leq \dfrac{KT^\alpha}{\Gamma(\alpha+1)\left[ 1-\left( \dfrac{M_2}{\theta}+M_3\right)\right]},
\end{align*}
which is the desired inequality \eqref{kk}
 \end{proof}
\begin{rem}Theorem \ref{Thm4.2}, gives the dependency of solution on  functions involved in the right hand side of the problem \eqref{Pe4.1}-\eqref{Pe4.2}. In particular, if $K=0$ then $f=g$, and we obatin uniquness of the solution. 
\end{rem}
 \section{Example}
Consider the nonlinear implicit FDEs of the form,
\begin{align}
^c_{0}\mathcal{D}_{t}^{\frac{1}{2}} x(t)&= \dfrac{\sqrt{\pi}}{4}-\dfrac{1}{2}t^{\frac{1}{2}} +\dfrac{1}{2}\left(x(t)+|^c_{0}\mathcal{D}_{t}^{\frac{1}{2}} x(t)|\right) , t\in J:=[0,0.5], \label{Pe5.1} \\
x(0)&=1\in X\label{Pe5.2}
\end{align}
in  the Banach space $X=(\mathbb{R},|\cdot|)$. Define $f:J\times X\times X\to X$ by $$f(t,x,y)=\dfrac{\sqrt{\pi}}{4}-\dfrac{1}{2}t^{\frac{1}{2}}+\dfrac{1}{2}\left(x+|y|\right)$$ 
For any $t_1,t_2\in J$ and $x_i,y_i\in X,\;i=1,2$, we have 
\begin{align*}
|f(t_1,x_1,y_1)-f(t_2,x_2,y_2)|&=\left|\left( \dfrac{\sqrt{\pi}}{4}-\dfrac{1}{2}t_1^{\frac{1}{2}} + \dfrac{1}{2}\left( x_1+|y_1|\right)\right) -\left( \dfrac{\sqrt{\pi}}{4}-\dfrac{1}{2}t_2^{\frac{1}{2}}+\dfrac{1}{2}\left(  x_2+|y_2|\right)\right) \right| \\
&\leq \dfrac{1}{2}|t_1^{\frac{1}{2}}-t_2^{\frac{1}{2}}|+\dfrac{1}{2}\left( |x_1-x_2|+||y_1|-|y_2||\right)\\
&\leq\dfrac{1}{2} |t_1-t_2|+\dfrac{1}{2} \left( |x_1-x_2|+|y_1-y_2|\right) 
\end{align*}
Therefore, the function $f$ satisfies hypothesis (A1) with constants  $M_1=M_2=M_3=\dfrac{1}{2}.$ 
Note that, 
$$
\dfrac{M_2 T^\alpha}{\Gamma(\alpha+1)}+M_3=\dfrac{\frac{1}{2}(0.5)^{\frac{1}{2}}}{\Gamma(\frac{3}{2})}+M_3=\sqrt{\dfrac{0.5}{\pi}}+\dfrac{1}{2}=0.8989<1.$$
Since the function $f$ satisfies all conditions of Theorem \ref{Thm3.2}, the problem \eqref{Pe5.1}--\eqref{Pe5.2} has unique solution. By direct calculation one can verify that the solution of the  problem \eqref{Pe5.1}--\eqref{Pe5.2} is  
$$
x^*(t)=t^{\frac{1}{2}}+\mathbb{E}_{\frac{1}{2}}(t^\frac{1}{2}), t\in J.
$$
Next, we consider another  implicit FDEs
\begin{align}
^c_{0}\mathcal{D}_{t}^{\frac{1}{2}} x(t)&= \dfrac{t^\frac{1}{2}}{2}+ +\dfrac{1}{2}\left(x(t)+|^c_{0}\mathcal{D}_{t}^{\frac{1}{2}} x(t)|\right) , t\in[0,0.5], \label{Pe5.3} \\
x(0)&=1-\dfrac{\sqrt{\pi}}{2}\in X\label{Pe5.4}
\end{align}
Define $g:J\times X\times X\to X$ by $$g(t,x,y)=\dfrac{t^\frac{1}{2}}{2}+ +\dfrac{1}{2}\left(x(t)+|^c_{0}\mathcal{D}_{t}^{\frac{1}{2}} x(t)|\right) , ~t\in J.$$ 
One can verify that the function $g:J\times X\times X\to X$ defined by $$
g(t,x,y)=\dfrac{t^\frac{1}{2}}{2}+ +\dfrac{1}{2}\left(x+|y|\right),
$$ 
satisfies all conditions of Theorem \ref{Thm3.2} with constants  $M_1=M_2=M_3=\dfrac{1}{2}.$  One can verify that 
$$
y^*(t)= t^{\frac{1}{2}}+\mathbb{E}_{\frac{1}{2}}(t^\frac{1}{2})-\dfrac{\sqrt{\pi}}{2} ,\; t\in J.
$$
is a exact solution of \eqref{Pe5.3}-\eqref{Pe5.4}. Note that for any $t\in J$ and $x,y\in X$
\begin{align*}
\left| f(t,x,y)-g(t,x,y)\right|&= \left| \left( \dfrac{\sqrt{\pi}}{4}-\dfrac{1}{2}t^{\frac{1}{2}}+\dfrac{1}{2}\left(x+|y|\right)\right) -\left( \dfrac{t^\frac{1}{2}}{2}+ +\dfrac{1}{2}\left(x+|y|\right)\right) \right| \\
&=\left| \dfrac{\sqrt{\pi}}{4}-t^\frac{1}{2}\right|\leq  \dfrac{\sqrt{\pi}}{4}+t^\frac{1}{2}:=\eta(t).
\end{align*} 
Since $f,g$ satisfies all the conditions of Theorem \ref{ThmP3.3}, by applying it, we obtain 
\begin{equation}\label{e5.5}
\|x^*-y^*\|_B\leq \|x_0-y_0\|+\dfrac{ K_\eta}{\theta\left[ 1-\left( \dfrac{M_2}{\theta}+M_3\right)\right] },
\end{equation} 
where $K_\eta=\max _{t\in J}\eta(t)=\max _{t\in J}\left( \dfrac{\sqrt{\pi}}{4}+t^\frac{1}{2}\right)=\dfrac{\sqrt{\pi}}{4}+(0.5)^\frac{1}{2} =1.1502.$ 
Take $\theta=2$, we have $\dfrac{M_2}{\theta}+M_3<1$. With this choice of  $\theta$, from \eqref{e5.5}, we have
 \begin{equation}\label{e56}
 \|x^*-y^*\|_B\leq |1-(1-\dfrac{\sqrt{\pi}}{2})|+\dfrac{ 1.1502}{2\left[ 1-0.75\right] }=\dfrac{\sqrt{\pi}}{2}+2.3004
 \end{equation}
Using the exact solution and by actual calculation, we have 
\begin{equation}\label{e57}
\|x^*-y^*\|_B=\sup_{t\in J}\dfrac{|x^*-y^*|}{\mathbb{E}_\frac{1}{2}(\theta t^{\frac{1}{2}})}\leq|x^*-y^*|=\dfrac{\sqrt{\pi}}{2}<\dfrac{\sqrt{\pi}}{2}+2.3004 
\end{equation}
\begin{rem}Form the inequalities \eqref{e56} and \eqref{e57}, it is observed that the difference of solution calculated  by applying  Theorem \ref{ThmP3.3} is similar to the difference of solution that calculated actually.
\end{rem} 
\section{Concluding remarks}
Observing the applicability of Caputo--Fabrizio (CF) fractional derivative operator \cite{CF1}-\cite{CF4}
and Atangana--Baleanu--Caputo (ABC) fractional derivative operator \cite{ABC1,ABC2,ABC3} that possesses a non-singular kernel, one can analyze and extend results of the present papers for the implicit FDEs with CF and ABC
fractional derivative operators.

\end{document}